\documentclass[11pt, sort&compress]{elsarticle} 

\usepackage{fullpage}
\usepackage{amsfonts} 
\usepackage{graphics}

\usepackage{here}

\usepackage{url}      

\usepackage{graphicx}
\usepackage{amsmath, amsthm, amssymb}
\usepackage{amssymb,amsmath,graphics,color}

\usepackage{multirow}

\usepackage{tocloft}

\allowdisplaybreaks[4] 

\newlength{\bibitemsep}
\newlength{\bibparskip}
\let\oldthebibliography\thebibliography
\renewcommand\thebibliography[1]{%
  \oldthebibliography{#1}%
  \setlength{\parskip}{\bibitemsep}%
  \setlength{\itemsep}{\bibparskip}%
}

\usepackage{refcheck} 
\norefnames	
\nocitenames 	
\ignoreunlbld 
\setoffmsgs 

\usepackage{xcolor}

\definecolor{darkbrown}{rgb}{.5,.1,.1} 
\usepackage{xcolor}

\def\futur #1{}

\def\emevder#1{}

\def\eme#1{}
\def\emeref#1{}

\def\opt#1{}

\newtheorem{thm}{Theorem}[]
\newtheorem{cor}[thm]{Corollary} 
\newtheorem{prop}[thm]{Proposition} 

\newtheorem{definition}[thm]{Definition}
\newtheorem{observation}[thm]{Observation}

\def\newline{\hfill\break}

\def\Int{\mathrm{Int}}

\def\Ext{\mathrm{Ext}}

\def\min{\mathrm{min}}

\def\ep{\varepsilon}
\def\io{\iota}

\def\s{\setminus}

\def\bs{\bigskip}

\def\ss{\smallskip}

\newcommand{\overbar}[1]{\mkern 1.5mu\overline{\mkern-1.5mu#1\mkern-1.5mu}\mkern 1.5mu}
\def\bar{\overbar}    

\def\ra{\rightarrow}

\def\M{M}
\def\N{N}

\title{On Tutte polynomial expansion formulas \\ 
in perspectives of matroids and oriented matroids}

\author[lirmm]{Emeric Gioan\corref{cor1}}
\ead{emeric.gioan@lirmm.fr}

\cortext[cor1]{Email: emeric.gioan@lirmm.fr}

\address[lirmm]{CNRS, LIRMM, Universit\'e de Montpellier, France}

\date{}

\begin{document}

\begin{abstract} 
%
%
%
We introduce the active partition of the ground set of an oriented matroid perspective (or quotient, or strong map) on a linearly ordered ground set. 
The reorientations obtained by arbitrarily reorienting parts of the active partition share the same active partition.
%
This yields an equivalence relation for the set of reorientations of an oriented matroid perspective, whose classes are enumerated by coefficients of the Tutte polynomial,
and a remarkable partition of the set of reorientations into boolean lattices, 
from which we get a short direct proof of a 4-variable expansion formula for the Tutte polynomial in terms of orientation activities.
This formula was given in the last unpublished preprint by Michel Las Vergnas; the above equivalence relation and notion of active partition generalize 
 a former construction  in oriented matroids by Michel Las Vergnas and the author; and the 
possibility
of such a 
proof technique in perspectives was announced in the aforementioned preprint.
We also briefly highlight how the 5-variable expansion of the Tutte polynomial in terms of subset activities in matroid perspectives 
comes in a similar 
way from the known partition of the power set of the ground set into boolean lattices related to subset activities 
(and we complete the proof with a property which was missing in the literature).\break
In particular, the paper applies to matroids and oriented matroids on a linearly ordered ground~set.
\end{abstract}

\maketitle   

\noindent\emph{Dedicated to Michel Las Vergnas' spirit.}


\emevder{faire corrections d'anglais, comme dans chapter, voir fichier anglais.tex  passage "SUR CHAPTER"}%

\section{Introduction}
\label{sec:intro}

The paper deals with (oriented) matroid perspectives on a linearly ordered ground set, with the Tutte polynomial in terms of subset activities, and with the Tutte polynomial in terms of orientation activities.
Let us situate these aspects with respect to the literature, and present the contributions of the paper. Formal definitions and general statements are given in the next sections.%
\bs

An ordered pair $(M,N)$ of (oriented) matroids forms a perspective  (also called morphism, or quotient, or strong map in the literature, up to unimportant variants), 
when they satisfy a certain structural relationship,
which is notably consistent with
linear maps in (real) vector spaces, and with (directed) graph homomorphisms in 
graphs.
%
See \cite{Ku86} or \cite[Section 7.3]{Ox92} (see also \cite{LV80,LV84}) in matroids, and  \cite[Section 7.7]{OM99} in oriented matroids.
\futur{ref chapter tutte}%
The Tutte polynomial of 
such a perspective 
has been defined in \cite{LV80, LV99} 
in terms of rank functions as:
$$t(M,N;x,y,z)\ =\  \sum_{A\subseteq E}\ (x-1)^{r(N)-r_{N}(A)}\ (y-1)^{\mid A\mid-r_M(A)} \ z^{r(M)-r(N)-\bigl(r_{M}(A)-r_{N}(A)\bigr)}.$$
Ever since, this polynomial has been mainly studied  by Michel Las Vergnas in a series of papers, including \cite{LV84, LV12, LV13}.
See  \cite{GiChapterPerspectives} for a recent 
survey.
In the particular case where $M=N$, one retrieves the usual Tutte polynomial of the matroid $M$.
\bs

The main novelty of the present paper concerns oriented matroid perspectives.
Let us give a glimpse 
in terms of 
a single oriented matroid $M$ on a linearly ordered set~$E$. The Tutte polynomial of $M$ is known to have the following expression \cite{LV84}:
$$t(M;x,y)\ =\ \sum_{A\subseteq E}
\ {\Bigl({x\over 2}\Bigr)}^{o^*(-_AM)}\ {\Bigl({y\over 2}\Bigr)}^{o(-_AM)}
$$
where $o^*(-_AM)$, resp. $o(-_AM)$, is the dual orientation activity, resp. orientation activity, of the oriented matroid $-_AM$, which counts smallest elements of positive cocircuits, resp. circuits.
See  
\cite{LV84, AB2-b} 
for interpretations and applications of these orientation activity parameters.
\futur{or chapter tutte om}%
In \cite{LV12}, a 4-variable expansion is given:
$$t(M;x+u,y+v)\ =\ \sum_{A\subseteq E} x^{\theta^*_M(A)}u^{\bar\theta^*_M(A)}y^{\theta_M(A)}v^{\bar\theta_M(A)}$$
where the parameters used for exponents refine orientation activities, depending on $M$ and $A$, 
not only on $-_AM$.
Actually, both \cite{LV84} and \cite{LV12} work at the level of oriented matroid perspectives (see general statements in Section \ref{sec:orientations}), and a detailed example is given in \cite{LV12}.
The proof proposed in \cite{LV12} is by deletion/contraction and is rather technical%
\footnote{Let us warn the reader that \cite{LV12} is an unpublished preliminary preprint, containing some inaccuracies. For instance, \cite[Equation (4) page 3]{LV12} is not correct (the parameters $cr$ and $nl$ are missing, see \cite{GoTr90,LV13}). More importantly, the definition of the refined active bijection  given at the end of \cite[page 20]{LV12} is not correct: it is not complete and given with a wrong parameter correspondence. The correct definition and correspondence are given in \cite{ABG2, AB2-b}, consistently with the definitions previously given in \cite{Gi02,  GiLVEurocomb07}.
}%
.
Also, it is suggested 
that an alternative proof could be obtained more structurally, using active partitions and activity classes of oriented matroids addressed in \cite{Gi02, GiLV05, ABG2, GiLVEurocomb07, AB2-b} (see also \cite{GiChapterOriented} for an overview).

In Section \ref{sec:orientations}, we give such a short direct proof.
Furthermore, we also give it at the level of oriented matroid perspectives. For that purpose, we first introduce the active partition of the ground set of an oriented matroid perspective on a linearly ordered ground set.
Then, we show that the reorientations obtained by arbitrarily reorienting parts of the active partition share the same active partition, 
a result which generalizes the same property known for oriented matroids.
This yields an equivalence relation for the set of reorientations,
whose classes are enumerated by coefficients of the Tutte polynomial,
and a remarkable structural partition of this set 
into boolean lattices called activity classes.
The enumerative counterpart of this partition turns out to be exactly the aforementioned  4-variable expansion formula for the Tutte polynomial.
Various corollaries are derived, such as enumerating remarkable subsets of reorientations by Tutte polynomial evaluations.%
\bs

As far as matroids are concerned, for a matroid $M$ on a linearly ordered set $E$, the Tutte polynomial has the two well-known following expressions \cite{Tu54, Cr69}:
%
$$t(M;x,y)\ =\ \sum_{A\subseteq E}\ (x-1)^{r(M)-r(A)}\ (y-1)^{|A|-r(A)}\ 
=\ \sum_{B\text{ basis}}\ x^{\io(B)}\ y^{\ep(B)}$$
where $\io(B)$, resp. $\ep(B)$, is the internal, resp. external, activity of the basis $B$.
A striking result is that those two expressions can be unified under the following 4-variable expansion:
$$t(M;x+u,y+v)\ =\ \sum_{A\subseteq E}\ x^{\io(A)}\ u^{r(M)-r(A)}\ y^{\ep(A)}\ v^{|A|-r(A)},$$
where the notions of activities are extended to subsets.
This result has been shown for matroids 
in \cite{GoTr90} (actually in a more general form).
It has been generalized  in \cite{LV13} to a 5-variable expansion formula the Tutte polynomial of a matroid perspective, recalled below as Theorem \ref{th:Tutte-4-variables-perspective}.
Let us mention that, using suitable variable specifications, one can obtain further expressions for the Tutte polynomial, see a list in  \cite{GoTr90} completed in~\cite{LV13}.

As precisely addressed in \cite{LV13}, the notion of subset activities is related to 
the following classical result: 
the set $2^E$ of subsets of $E$ is partitioned into boolean intervals of type
$[\ B\setminus \Int(B),\ B\cup \Ext(B)\ ]$,
where $B$ ranges over all bases of $M$ (and where $\Int(B)$, resp. $\Ext(B)$, denotes the set of internally, resp. externally, active elements w.r.t. $B$).
This result, known for matroids from 
\cite{Cr69},
%
generalizes to matroid perspectives. 
Actually, it can be obtained using an even more general combinatorial construction \cite{Da81}, which applies to matroids and matroid perspectives as well, as detailed in \cite{LV13} (see also 
\cite{Bj92, LV01} for related results in matroids, 
\cite{GoMM97} for an\break 
independent generalization to greedoids, and \cite{BrGONo09} for an inverse property in matroids).%

In Section \ref{sec:subsets}, 
placed at the level of matroid perspectives,
we sum up this setting and highlight 
how 
the above expansion formula
can be directly derived as the enumerative counterpart of
the aforementioned fundamental structural partition of the power set into intervals.
The proof 
in \cite{LV13} uses this partition in an indirect way, passing through Tutte polynomial derivatives. 
Furthermore, a property of this partition is missing in \cite{LV13}, though  required for the proof and implicitly used.
Here, in  Proposition~\ref{lem:rcd}, we explicitly state and prove this property, hence completing~\cite{LV13}.

\bs

Finally, the point of the paper is to complete \cite{LV12,LV13} and to show how
each Tutte polynomial expansion formula is the enumerative counterpart of a remarkable structural partition of the power set into boolean lattices, arising  as soon as the ground set is linearly ordered.
The 
initial formula enumerates the boolean lattices each as an individual part, and the expansion
formula with two more variables enumerates in addition all subsets within these boolean lattices.

From the stuctural viewpoint, let us mention that active partitions in oriented matroids have further remarkable properties, and that the 
active bijection transforms the formulas in terms of orientation and subset activities into each other, while transforming the boolean lattices of the two constructions into each other, e.g. 
see \cite{Gi02, GiLV05, ABG2, GiLVEurocomb07, AB2-b}.
\futur{tutte chapter? ab2-b d'abord seule?}%
The generalization of such properties and bijections to oriented matroid perspectives is work in progress.


%




\section{Building on orientation activities in oriented matroid perspectives}
\label{sec:orientations}


The reader is referred to \cite{OM99} for background on oriented matroids, and in particular to \cite[Section 7.7]{OM99} for details on oriented matroid perspectives.
%
%
Unless it is ambiguous, we usually 
use
the same notation for a signed set and its support.
%
Let us recall that two signed sets are conformal if they have the same signs on their intersection, and that a vector, resp. covector, 
in an oriented matroid can be defined as a conformal union of circuits, resp. cocircuits.
An ordered pair $(M,N)$ of oriented matroids on the same ground set form an \emph{oriented matroid perspective} (also called \emph{strong map} or \emph{quotient map}), denoted $M\ra N$, if one of the following equivalent properties~holds%
\footnote{%
We choose to denote $M\ra N$ rather than $M\ra M'$ which is commonly used in  references on (oriented) matroid perspectives (such as \cite{LV12,LV13}) in order to avoid  confusions with the prime used in the notations for the cyclic parts (involving orientation-active elements and complementaries of flats of the dual) in the active decomposition of a single oriented matroid.
Indeed, the present paper mixes the two settings, involving two single oriented matroids and the active decomposition of each of them
(see details and references below Definition \ref{def:active-partition}).
 }%
:
\begin{itemize}
\itemsep=0mm
\partopsep=0mm 
\topsep=0mm 
\parsep=0mm
\item any circuit of $M$ is a vector of $N$;
\item any vector of $M$ is a vector of $N$;
\item any cocircuit of $N$ is a covector of $M$;
\item any covector of $N$ is a covector of $M$;
\item no circuit of $M$ and cocircuit of $N$ have a non-empty conformal intersection%
\footnote{%
This characterization is mentioned in \cite{LV12}, but not in the usual reference on oriented matroid theory \cite{OM99}. 
It is direct since vectors are signed subsets orthogonal to all cocircuits:
a signed subset $C$ (meant to be a circuit of $M$) is not a vector of $N$ if and only if $C$ is not orthogonal to some cocircuit of $N$, that is if and only if it has a non-empty conformal intersection with some cocircuit of $N$.
}%
.
\end{itemize}
In particular, for an oriented matroid $M$, 
$M\ra M$ is an oriented matroid perspective.
Actually, if $M\ra N$ is an oriented matroid perspective and $r(M)=r(N)$ then $M=N$.
Moreover,
if $M\ra N$ is an oriented matroid perspective then $N^*\ra M^*$ is also an oriented matroid perspective.
Let us mention that the underlying matroids form a matroid perspective, and that, in contrast, an oriented matroid perspective 
is not equivalent to
 an extension followed by a contraction (see Section~\ref{sec:subsets}).

%

\bs

%
Let $M$ be an oriented matroid on a linearly ordered set $E$.
Following {\cite{LV84}}, we denote
\vspace{-0.2cm}
\begin{eqnarray*}
O(M) & =& \bigl\{\ \min(C)\mid \text{ $C$ positive circuit of $M$}\ \bigr\}, \\
O^*(M) & =& \bigl\{\ \min(C)\mid \text{ $C$ positive cocircuit of $M$}\ \bigr\}. 
\end{eqnarray*}
The sets $O(M)$ and $O^*(M)$ are respectively known as the sets of \emph{active} and \emph{dual-active} elements of $M$.
The parameters $| O(M) |$ and $| O^*(M) |$ are respectively known as the \emph{(orientation) activity} and \emph{dual (orientation) activity}  of $M$.
%
%
%
%
%
%
%
%
One can observe that $O^*(M)=O(M^*)$, that $O(M)\cap O^*(M)=\emptyset$, that $O(M)=\emptyset$ if and only if $M$ is acyclic, and that $O^*(M)=\emptyset$ if and only if $M$ is totally cyclic.
%
\bs

Let $M\ra N$ be an oriented matroid perspective on a linearly ordered set $E$. 
By {\cite{LV84}}, 
its Tutte polynomial satisfies
$$t(M,N;x,y,1)=\sum_{A\subseteq E}
\ {\Bigl({x\over 2}\Bigr)}^{| O^*(-_AN) |}\ {\Bigl({y\over 2}\Bigr)}^{| O(-_AM) |}.
$$

%
Note that, in particular, $t(M,N;0,0,1)$ is equal to the number of subsets $A$ of $E$ such that $-_AM$ is acyclic and $-_AN$ is totally cyclic
(see \cite{LV84}\futur{ref chapter} 
for some interpretations of this special evaluation).
%
%
For such an oriented matroid perspective $M\ra N$, 
the sets $O(M)$ and $O^*(N)$ will be respectively called, by extension of the terminology, the sets of \emph{active} and \emph{dual-active} elements of $M\ra N$. 

%

\bs

Now, we  introduce a new construction, generalizing  to oriented matroid perspectives a construction in oriented matroids detailed in \cite{Gi02,AB2-b} 
(see also \cite{GiLV05,ABG2} in graphs).
We refer the reader to these references for geometrical interpretations, that can readily be adapted in the setting of oriented matroid perspectives.
\bs



Let $M\ra N$ be an oriented matroid perspective on a linearly ordered set $E$
with%
\begin{eqnarray*}
O(M) & =& \{g_1, \ldots, g_\ep\}_<, \\
O^*(N) & =& \{h_1, \ldots, h_\io\}_<. 
\end{eqnarray*}
We define $G_\ep=\emptyset$ and, for  $0\leq k\leq\ep-1$, we define
$$G_{k}\ =\ \bigcup_{\substack{ C \hbox{\small\ positive circuit of } M\\{\min(C)\ \geq\ g_{k+1}}}}C.$$
We define $H_\io=E$ and, for  $0\leq k\leq\io-1$, we define
$$H_k\ =\ E\ \s \ \bigcup_{\substack{ C \hbox{\small\ positive cocircuit of } N\\{\min(C)\ \geq\ h_{k+1}}}}C.$$
%
%
Observe that $G_0$ is the union of positive circuits of $M$ (i.e. the maximal positive vector of $M$), $E\s G_0$ is the union of positive cocircuits of $M$ (i.e. the maximal positive covector of $M$), $H_0$ is the union of positive circuits of $N$ (i.e. the maximal positive vector of $N$), and $E\s H_0$ is the union of positive cocircuits of $N$ (i.e. the maximal positive covector of $N$).
Hence, by definition of an oriented matroid perspective, we have
$$G_0\ \subseteq\ H_0.$$
Observe that we can have $G_0= H_0$, e.g. when $M=N$, and that we can also have \hbox{$\emptyset=G_0\subset H_0=E$}, when $M$ is acyclic and $N$ totally cyclic, as for the reorientations mentioned above counted by $t(M,N;0,0,1)$.

\begin{definition}
\label{def:active-partition}
We call \emph{active filtration of $M\ra N$} the sequence of nested  subsets
$$\emptyset=G_\ep\ \subset\  \ldots\ \subset\  G_1\ \subset\  G_0\ \subseteq\  H_0\ \subset\   H_1\  \subset\  \ldots\ \subset\  H_\io = E.$$
And we call \emph{active partition of $M\ra N$} the partition of $E$ given by
$$E\ =\  (G_{\ep-1}\s G_\ep)\ \uplus\ \dots \ \uplus\ (G_{0}\s G_1)\ \uplus\ (H_{0}\s G_0)\ \uplus\ (H_1\s H_0)\ \uplus\ \dots \ \uplus\ (H_{\io}\s H_{\io-1}).$$
The parts contained in $G_0$ are called the \emph{cyclic parts}, 
the parts contained in $E\s H_0$ are called the \emph{acyclic parts}, 
and the part $H_{0}\s G_0$ (possibly empty) is called the \emph{hybrid part}.
We assume that the active filtration and active partition are always given along with the sets $G_0$ and $H_0$ (hence one can deduce the active filtration from the active partition and vice versa). 
\end{definition}

The above definition generalizes
active filtrations and active partitions of oriented matroids on a linearly ordered ground set, as defined in \cite{Gi02, AB2-b} (see also \cite{GiLV05,ABG2} in graphs). When $M=N$ the definitions coincide, and, in general, the acyclic parts for $M\ra N$ equal the acyclic parts for $N$, and the cyclic parts for $M\ra N$ equal the cyclic parts for~$M$.
%
\ss

Moreover, note that we have,
for $1\leq k\leq \ep$,  
\vspace {-2mm}
\begin{eqnarray*}
g_k&=&\min(\ G_{k-1}\setminus G_{k}\ ), \\[2mm]
G_{k-1}\setminus G_{k}&=& 
\bigcup_{\substack{ C \hbox{\small\ positive circuit of } M\\{\min(C)\ =\ g_{k}}}}C
\ \ \setminus\ 
\bigcup_{\substack{ C \hbox{\small\ positive circuit of } M\\{\min(C)\ >\ g_{k}}}}C,
\end{eqnarray*}
and, for $1\leq k\leq\io$, 
\vspace {-2mm}
\begin{eqnarray*}
h_k&=&\min(\ H_{k}\setminus H_{k-1}\ ),\\[2mm]
H_{k}\setminus H_{k-1}&=&
\bigcup_{\substack{ C \hbox{\small\ positive cocircuit of } N\\{\min(C)\ =\ h_{k}}}}C
\ \ \setminus\ 
\bigcup_{\substack{ C \hbox{\small\ positive cocircuit of } N\\{\min(C)\ >\ h_{k}}}}C.
\end{eqnarray*}
\begin{prop}
\label{thm:act-classes}
\futur{proposition ou theorem ?}%
Let $M\ra N$ be an oriented matroid perspective on a linearly ordered set $E$.
Let $A\subseteq E$ be any union of parts of the active partition of $M\ra N$. 
Then the oriented matroid perspective $-_AM\ra -_AN$ has the same active partition as $M\ra N$ (meaning precisely the same cyclic, acyclic and hybrid parts).
%
\end{prop}

\begin{proof}
The proof is 
fairly simple,
but let us take care to detail it. 
We will use the following notation, with the abuse of identifying a positive signed subset of $E$ with the subset of $E$ which is its support.
For $e\in E$,  $V(M;e)$ denotes the union of all positive circuits of $M$ with smallest element greater than or equal to $e$.
Equivalently, $V(M;e)$  is the largest (for inclusion) positive vector with smallest element greater than or equal to $e$ (obtained by conformal composition of positive circuits).

First, we consider the case where 
$A=V(M;a)$ for 
a given element $a\in O(M)$
(using the previous notations, it means that we have $a=g_{k+1}$ and $A=G_k$ for some $0\leq k\leq \ep-1$).
Obviously, $A$ is also a positive vector of~$-_AM$.
Let us fix  $e\in E$ and let us prove that $V(M;e)=V(-_AM;e)$.


\futur{a t'on vraiment besoin de distinguer $e<a$ et $e>a $ ?}%

Assume that $e\geq a$. 
Let $C$ be a  positive circuit of $M$, with $\min(C)\geq e$. 
Then $C\subseteq A$ by definition of $A$, then $C$ is also a positive circuit of $-_A M$. 
%
Conversely, let $C$ be a  positive circuit of $-_AM$, with $\min(C)\geq e$. 
Then the signed subset $A\circ C$ is a positive vector in $-_AM$. Then, obviously,
this signed subset $A\circ C$ is also a positive vector in $M$. 
This positive vector of $M$ has smallest element $a$, hence it is contained in $A$ by definition of $A$. So we have $C\subseteq A$, and so $C$ is also a positive circuit of $M$.
Finally, when $e\geq a$, $M$ and $-_AM$ have exactly the same positive circuits with smallest element greater than or equal to $e$, hence $V(M;e)=V(-_AM;e)$.
\emevder{XXXX le cas $e\geq a$ de ce paragraphe est-il utile ? ne peut on out traiter d'un coup ? XXXXX}

Assume now that $e<a$. 
Let $C$ be a  positive circuit of $\M$, with $\min(C)\geq e$. Since $A$ is a positive vector of $-_AM$ and $C$ is positive in $-_AM$ on $C\s A$, we have that $A\circ C$ is a positive vector of $-_AM$ with smallest element greater than or equal to $e$. 
This vector can be obtained in $-_AM$ by a conformal composition of positive circuits, so $A\cup C$ is contained in 
$V(-_AM;e)$.
So $V(M;e)\subseteq V(-_AM;e)$.
%
%
Conversely, let $C$ be a  positive circuit of $-_AM$, with $\min(C)\geq e$. 
By the same reasoning, we have that $A\circ C$ is a positive vector of $M$ with smallest element greater than or equal to $e$, that $A\cup C$ is contained in 
$V(M;e)$,
and that  $V(-_AM;e)\subseteq V(M;e)$.

Finally, we have proved that, for every $e\in E$, we have $V(M;e)= V(-_AM;e)$.
 %
%
%
This implies that $O(M)=O(-_AM)$ and, by definition of the active partition, that the cyclic parts of the active partition of $M\ra N$ are equal to the cyclic parts of the active partition of $-_AM\ra -_AN$
(precisley, using the previous notations, we have proved that, for $0\leq k\leq \ep-1$, the sets $G_k$ are the same when built in $M$ or in $-_AM$).

Since $A$ is a union of positive circuits of $M$ and since $M\ra N$ is an oriented perspective, we have that $A$ is a positive vector of $N$. Hence $A$ does not interfere with positive cocircuits of $N$, which remain exactly the same in $-_AN$. Precisely, the acyclic parts of the active partition of  $M\ra N$ are obviously equal to the acyclic parts of the active partition of $-_AM\ra -_AN$.
And so the hybrid part also has to be the same in the active partitions of $M\ra N$ and $-_AM\ra -_AN$.

So we have proved the theorem in the case where
$A=V(M;a)$ for  $a\in O(M)$.
Using the previous notations for the active partition of $M\ra N$, we have proved that $M\ra N$ and $-_AM\ra -_AN$ have the same active partitions when $A=G_k$ for any $0\leq k\leq \ep-1$ (hence for $\ep$ possible non-empty sets $A$).
Then, by transitivity, applying the result for $A=G_k$ and $A=G_{k+1}$, 
we have that $M\ra N$ and $-_AM\ra -_AN$ have the same active partitions when $A=G_k\s G_{k+1}$ for any $0\leq k\leq \ep-1$.
Then, by transitivity again, we have proved the theorem when $A$ is any union of cyclic parts of the active partition of $M\ra N$ (hence for $2^\ep$ possible sets $A$). 

Now, we can use the same reasoning to cocircuits in $N$ and acyclic parts of the active partition of $M\ra N$, which is equivalent to apply the above result in the dual oriented matroid perspective $N^*\ra M^*$.
We obtain that the theorem is true when $A$ is any union of acyclic parts of the active partition of $M\ra N$ (hence for $2^\io$ possible sets $A$). 
By transtivity, combining the two dual results, we have that
the theorem is true when $A$ is any union of cyclic or acyclic parts of the active partition of $M\ra N$ (hence for $2^{\io+\ep}$ possible sets $A$).

Finally, it remains to handle the case where the hybrid part is non-empty.
In this case, as the active partition of $M\ra N$ is obviously the same as the active partition of $-_EM\ra -_E N$, we have that the
theorem is true when $A$ is any union of cyclic or acyclic parts of the active partition of $M\ra N$, or its complementary (hence for $2^{\io+\ep+1}$ possible sets $A$).
This means that $A$ can be  any union of parts of the active partition of $M\ra N$ (cyclic, acyclic, or hybrid).
\end{proof}

Let $M\ra N$ be an oriented matroid perspective on a linearly ordered set $E$.
The set of all $-_AM\ra -_AN$ for $A\subseteq E$ is called the set of \emph{reorientations of $M\ra N$}, being understood that
this set is isomorphic to $2^E$. Note the slight abuse of notations,
as we distinguish for instance between $-_AM\ra -_AN$ and $-_{E\setminus A}M\ra -_{E\setminus A}N$ as reorientations of $M\ra N$, even though 
the two resulting 
oriented matroid perspectives are equal.
%

\begin{definition}
\label{def:act-classes}
Let $M\ra N$ be an oriented matroid perspective on a linearly ordered set $E$.
We define \emph{the activity class of $M\ra N$} as the set of 
its reorientations $-_AM\ra -_AN$ 
where 
$A$ is any union of 
cyclic or acyclic parts of its active partition. 
%
%
Observe that we forbid reorienting the hybrid part%
\footnote{This is done in order to distinguish reorientations from their complementaries when the hybrid part is non-empty. This yields tighter classes, closer to Tutte polynomial properties, and meaningful notably in the proof of Theorem~\ref{th:expansion-reorientations}.}%
.
%
By Proposition \ref{thm:act-classes} and 
extending this definition to all reorientations of  $M\ra N$,
we define \emph{activity classes of reorientations of $M\ra N$}, that partition the set of reorientations of $M\ra N$, and can be seen as classes of the equivalence relation defined by reorienting parts of the active partition.
%
\end{definition}
%
%

The above partition  of the set of reorientations of $M\ra N$ can be  summed up
as follows:
$$2^E\ \sim\ \biguplus_{\substack{\text{activity classes of }M \ra N\\ \text{\tiny (one $-_AM\ra -_AN$ chosen in each class)}}}
\Biggl\{\  2^{|O(-_AM)|+|O^*(-_AN)|} \ \
\substack{\text{\small reorientations obtained by}\\
\text{\small  reorienting cyclic and acyclic parts}\\
\text{\small  of the active partition of $-_AM\ra -_AN$}
}
\ \Biggr\}.$$
For practical purposes, it is also convenient to transpose this equivalence relation in terms of subsets $A\subset E$. 
In this setting, we can identify a reorientation $-_AM\ra -_AN$ with the subset $A\subseteq E$ which defines it. 
Let us fix  $A\subseteq E$. 
The active partition of $-_A\M\ra -_AN$
can be shortly denoted as:
$$E\ =\ A_0\ \cup\ \biguplus_{a\; \in\; O(-_A\M)\; \cup\; O^*(-_AN)}A_a$$
where $A_0$ is the hybrid part, and where $A_a$ is either a cyclic or an acyclic part, with $a=\min(A_a)$ for $a\in O(-_A\M)\cup O^*(-_AN)$.
Then, 
the activity class of  $-_A\M\ra -_AN$ can be renamed
the \emph{activity class of $A\subseteq E$ w.r.t. $M\ra N$} and denoted%
\footnote{The choice of letters $P$ and $Q$, in that order, respectively associated   with dual-activity in $N$ and activity in $M$, might seem non-natural, but it is  deliberate for the sake of matching with variables $x$ and $y$ respectively (and with internal and external subset activities respectively, as in Section \ref{sec:subsets} and in oriented matroids \cite{ABG2,AB2-b,LV13}).}%
:
$$cl_{M\ra N}(A)=\Biggl\{ \ A\ \triangle\ \Bigl(\ \biguplus_{a\in P\cup Q}A_a\ \Bigr)\ \ \mid \ \ P\subseteq O^*(-_AN), \ Q\subseteq O(-_AM)\ \Biggr\}.$$
Then, the different sets $cl_{M\ra N}(A)$, for $A\in 2^E$, form a partition of $2^E$:
$$2^E=\biguplus_{\text{\scriptsize (one $A$ chosen in each $cl_{M\ra N}(A)$)}}
cl_{M\ra N}(A)$$
Observe  that each $cl_{M\ra N}(A)$ has a boolean lattice structure, yielding a partition of $2^E$ into boolean lattices.
The above definition and construction generalize
activivity classes of oriented matroids on a linearly ordered ground set, as defined in \cite{Gi02, AB2-b} (see also \cite{GiLV05,ABG2} in graphs). 
\bs

Now, let us follow \cite{LV12} and refine reorientation activities into four parameters,
applied to reorientations of a given reference oriented matroid $\M$.

\begin{definition} 
\label{def:gene-act-ori}
Let $\M$ be an oriented matroid on a linearly ordered set $E$ and $A\subseteq E$. We define:
\vspace{-3mm}
\begin{eqnarray*}
   \Theta_\M(A)&=&O(-_A\M)\s A,  \\
 \bar\Theta_\M(A)&=&O(-_A\M)\cap A, \\ 
   \Theta^*_\M(A)&=&O^*(-_A\M)\s A, \\
\bar\Theta^*_\M(A)&=&O^*(-_A\M)\cap A. 
\end{eqnarray*}
Hence we have $O(-_A\M)=\Theta_\M(A)\uplus\bar\Theta_\M(A)$ 
and (dually) $O^*(-_A\M)=\Theta^*_\M(A)\uplus\bar\Theta^*_\M(A)$.
\end{definition}

\futur{reorganiser ce passage? prop sur unicite de active-fixe?}%
Actually, in oriented matroids and oriented matroid perspectives as well, the  parameters of Definition~\ref{def:gene-act-ori}
situate any reorientation in its activity class seen as a boolean lattice. Let us detail this fundamental feature (in a similar way than done in \cite{ABG2,AB2-b}).
%

\begin{definition} 
\label{def:active-fixed}
Let $\M$ be an oriented matroid on a linearly ordered set $E$ and $A\subseteq E$.
The reorientation $-_AM$ of $M$ is called \emph{active-fixed} if no active element of $-_AM$ is reoriented w.r.t. $M$, that is if $\bar\Theta_\M(A)=O(-_A\M)\cap A=\emptyset$.
It is called \emph{dual-active-fixed} if no dual-active element of $-_AM$ is reoriented w.r.t. $M$, that is if $\bar\Theta^*_\M(A)=O^*(-_A\M)\cap A=\emptyset$.

For an oriented matroid perspective $M\ra N$ on a linearly ordered set $E$ and $A\subseteq E$, by extension of the terminology,
the reorientation $-_AM\ra -_AN$ will be called \emph{active-fixed} if $-_AM$ is active-fixed (w.r.t. $M$)
and \emph{dual-active-fixed} if $-_AN$ is dual active fixed (w.r.t. $N$).
\end{definition}


\begin{observation}
\label{obs:unique-active-fixed}
Let $M\ra N$ be an oriented matroid perspective on a linearly ordered set $E$.
In each activity class of reorientations of  $M\ra N$, 
there is one and only one reorientation $-_AM\ra -_AN$
which is active-fixed and dual-active-fixed
 (that is such that all active elements of $M$ and dual-active elements of $N$ are not reoriented with respect to the initial 
 $\M\ra N$).
\end{observation}

In other words, 
continuing the above setting, in the activity class $cl_{M\ra N}(A)$, 
one can choose as representative  the subset $A$ such that $$\bar\Theta_\M(A)=O(-_A\M)\cap A=\emptyset\ \ \text{ and }\ \ \bar\Theta^*_\N(A)=O^*(-_A\N)\cap A=\emptyset.$$
So, for this special $A$ and according to the above definition of $cl_{M\ra N}(A)$, an element $A'$ of $cl_{M\ra N}(A)$
is situated in $cl_{M\ra N}(A)$ by $P\subseteq  O^*(-_{A}\N)=O^*(-_{A'}\N)$ and  $Q\subseteq O(-_{A}\M)=O(-_{A'}\M)$ such that%
:
\futur{permuter ordre dans toutes les listes et enonces? d'abord dual/x/P/N puis primal/y/Q/M ?}
\futur{ceci pourrait etre seulement dans la preuve du thm...?}%
\begin{center}
   $
   \begin{array}{lclcl}
\Theta_\M(A')&=&O(-_{A}\M)\s Q,\\
\bar\Theta_\M(A')&=&Q,\\
\Theta^*_\N(A')&=&O^*(-_{A}\N)\s P,\\
\bar\Theta^*_\N(A')&=&P.
   \end{array}
   $
\end{center}


Now, we get the following Theorem \ref{th:expansion-reorientations}, a four variable expression of the Tutte polynomial in terms of the above four parameters, stated and proved by deletion/contraction in \cite{LV12}.
Here, we give a short natural proof of this theorem, using the Tutte polynomial formula in terms of orientation activities, the partition of the set of reorientations into activity classes, and their boolean lattice structure as discussed above.

%
%
%
%
%
%

\begin{thm}[\cite{LV12}]
\label{th:expansion-reorientations}
Let $M\ra N$ be an oriented matroid perspective on a linearly ordered set $E$. We have
\begin{large}
$$t(M,N;x+u,y+v,1)\ =\ \sum_{A\subseteq E} \ x^{|\Theta^*_\N(A)|}\ u^{|\bar\Theta^*_\N(A)|}\ y^{|\Theta_\M(A)|}\ v^{|\bar\Theta_\M(A)|}.$$
\end{large}
\end{thm}

\begin{proof}

Let us denote  
$[Exp] = \sum_{A\subseteq E}\ x^{|\Theta^*_\N(A)|}\ u^{|\bar\Theta^*_\N(A)|}\ y^{|\Theta_\M(A)|}\ v^{|\bar\Theta_\M(A)|}.$
Since $2^E$ is partitioned into activity classes 
by Definition \ref{def:act-classes},
and denoting $cl(A)$ for short instead of $cl_{M\ra N}(A)$, we have:
\vspace{-2mm}
\begin{eqnarray*}
[Exp] &= & \sum_{\substack{cl(A)\\ 
\text{one $A$ for each class}
}} \ \ \sum_{A'\in cl(A)}  \ x^{|\Theta^*_\N(A')|}\ u^{|\bar\Theta^*_\N(A')|}\ y^{|\Theta_\M(A')|}\ v^{|\bar\Theta_\M(A')|}
\end{eqnarray*}
\vspace{-3mm}

As discussed above, one can choose a representative $A$ of $cl(A)$ such that $O(-_AM)\cap A=\emptyset$ and $O^*(-_AN)\cap A=\emptyset$, 
and use the 
structure of $cl(A)$, to reformulate the expression as follows:
\futur{detailler preuve comme en commentaire dessous?}%
%
%
\vspace{-2mm}
\begin{eqnarray*}
%
[Exp] & = & \sum_{\substack{cl(A)\\ 
\text{one $A$ for each class}
}} \ \ \sum_{\substack{P\subseteq O^*(-_AN)\\Q\subseteq O(-_AM)}} x^{\mid O^*(-_AN)\s P\mid}u^{\mid P\mid}y^{\mid O(-_AM)\s Q\mid}v^{\mid Q\mid}\\
\end{eqnarray*}

\vspace{-3mm}
Now we can simply rewrite this expression as follows and use the binomial formula: 
\vspace{-2mm}
\begin{eqnarray*}
[Exp]&=&  
\sum_{\substack{cl(A)\\ 
\text{one $A$ for each class}
}}
\Biggl(\sum_{P\subseteq O^*(-_AN)} x^{\mid O^*(-_AN)\s P\mid}u^{\mid P\mid}\Biggr)
\Biggl(\sum_{Q\subseteq O(-_AM)} y^{\mid O(-_AM)\s Q\mid}v^{\mid Q\mid}\Biggr)\\
&=&  
\sum_{\substack{cl(A)\\ 
\text{one $A$ for each class}
}}
(x+u)^{\mid O^*(-_AN)\mid}\ (y+v)^{\mid O(-_AM)\mid}\\
\end{eqnarray*}

\vspace{-3mm}

For a better readability, let us denote $o(-_AM)$ for $|O(-_AM)|$, and $o^*(-_AM)$ for $|O^*(-_AM)|$.
Since the activity class $cl(A)$ has ${2^{o(-_AM)+o^*(-_AN)}}$ elements with the same activities, 
we have:
\vspace{-2mm}
\begin{eqnarray*}
[Exp] & = & 
%
\sum_{\substack{cl(A)\\ 
\text{one $A$ for each class}
}}
{1\over{2^{o^*(-_AN)+o(-_AM)}}} \sum_{A'\in cl(A)} (x+u)^{o^*(-_{A'}N)}\ (y+v)^{o(-_{A'}M)}\\
&=& 
\sum_{\substack{cl(A)\\ 
\text{one $A$ for each class}
}}
  \sum_{A'\in cl(A)} {\Bigl({x+u\over 2}\Bigr)}^{o^*(-_{A'}N)}\ {\Bigl({y+v\over 2}\Bigr)}^{o(-_{A'}M)}\\
&=&  \sum_{A\subseteq E} \ {\Bigl({x+u\over 2}\Bigr)}^{o^*(-_{A}N)}\ {\Bigl({y+v\over 2}\Bigr)}^{o(-_{A}M)}\\
&=&  t(M,N;x+u,y+v,1)
\end{eqnarray*}
by the formula in terms of orientation activities from \cite{LV84} recalled earlier.
\end{proof}


Numerous Tutte polynomial formulas, for oriented matroid perspectives or for oriented matroids as well, can be directly obtained from Theorem \ref{th:expansion-reorientations}, for instance by replacing variables ($x$, $u$, $y$, $v$) with $(x/2,x/2,y/2,y/2)$, 
or $(x+1,-1,y+1,-1)$, 
or $(2,0,0,0)$...
Let us give examples of such formulas in the corollary below. 
%
These formulas, and a detailed example for this set of formulas, are given  \cite{LV12}.
Actually, a total of 25 expansions, 9 different up to reordering, could be
thus obtained (in a similar way than in \cite{GoTr90, LV13} for formulas from subset activities). \emevder{A VERIFIER ?!?}



\begin{cor}[\cite{LV12}] 
Let us denote $\theta_\M(A)$ instead of $|\Theta_\M(A)|$, etc.
Let $p,q$ be two non-negative integers. 
Let $M\ra N$ be an oriented matroid perspective on a linearly ordered set $E$.
%
%
We have (among possible variants as discussed above):
\begin{large}
\begin{align*}
t(M,N;x,y,1)&&=&& \sum_{A\subseteq E}\ (x-1)^{\theta^*_{\N}(A)}(y-1)^{\theta_\M(A)} 
 &&=&& \sum_{\substack{A\subseteq E\\ \bar\theta^*_\N(A)=0\\ \bar\theta_\M(A)=0}}\ x^{\theta^*_{\N}(A)}y^{\theta_\M(A)} \\
t(M,N;0,0,1)&&=&&\sum_{A\subseteq E} (-1)^{\theta^*_{N}(A)+\theta_{\M}(A)}\\
{\partial^{p+q}t\over \partial x^p\partial y^q}(M,N;x,y,1)&&=&&\ p!q!\ 
\sum_{\substack{A\subseteq E\\ \bar\theta^*_{N}(A)=p\\ \bar\theta_M(A)=q}}\ x^{\theta^*_{N}(A)}y^{\theta_M(A)}.
\end{align*}
\end{large}

\vspace{-3mm}
\noindent Let $M$ be an oriented matroid on a linearly ordered set $E$. We have:
%
\vspace{-1mm}
\begin{large}
\begin{align*}
t(M;x+u,y+v)&&=&&\sum_{A\subseteq E} x^{\theta^*_M(A)}u^{\bar\theta^*_M(A)}y^{\theta_M(A)}v^{\bar\theta_M(A)}\\
t(M;x,y)&&=&& \sum_{A\subseteq E}\ (x-1)^{\theta^*_{\M}(A)}(y-1)^{\theta_\M(A)} 
 &&=&& \sum_{\substack{A\subseteq E\\ \bar\theta^*_\M(A)=0\\ \bar\theta_\M(A)=0}}\ x^{\theta^*_{\M}(A)}y^{\theta_\M(A)} \\
t(M;2,0)&&=&&\sum_{\substack{A\subseteq E\\ \bar\theta^*_\M(A)=\theta_\M(A)=\bar\theta_\M(A)=0}}2^{\theta^*_\M(A)}
&&=&&\sum_{\substack{A\subseteq E\\ \theta^*_\M(A)=\theta_\M(A)=\bar\theta_\M(A)=0}}2^{\bar\theta^*_\M(A)} \\
{\partial^{p+q}\over \partial x^p\partial y^q}\ t(M;x,y)&&=&& \ p!q!\ 
\sum_{\substack{A\subseteq E\\ \bar\theta^*_{M}(A)=p\\ \bar\theta_M(A)=q}}\ x^{\theta^*_{M}(A)}y^{\theta_M(A)}.  && &&
\end{align*}
\end{large}
\end{cor}

\begin{proof}
Direct by Theorem \ref{th:expansion-reorientations},
use Taylor's theorem for the derivative expressions. 
\emevder{XXXX verifier les expressions !!! XXXX}
\end{proof}

\begin{observation}
We have $\bar\theta^*_\N(A)=p$ and $\bar\theta_\M(A)=q$ if  $q$ active elements of $-_AM$  and $p$  dual-active elements of $-_AN$  are reoriented in $-_AM\ra-_AN$ w.r.t. $M\ra N$ (or equally  w.r.t. the unique reorientation of $M\ra N$ in the same activity class as $-_AM\ra-_AN$
which is active-fixed  and dual-active-fixed).
\emevder{A REVERFIEIR ! je crois avoir bien verifie, mais aussi je crains de m'emboruiller... peut-etre quec ete pbservation evidente alourdit inutilement ? ou bien clarifie l'usage de ces paramettres numeriques ? A TRANCHER !}%
\end{observation}

Let us end with a proposition summing up how the Tutte polynomial coefficients and evaluations count sets of reorientations based on activity classes and their representatives. It is direct from the previous definitions and results.
Let us mention that these reorientations are addressed in oriented matroids and bijectively related to certain subsets in terms of subset activities in \cite{ABG2, AB2-b}.

\begin{prop}
\label{prop:enum-classes}
Let $M\ra N$ be an oriented matroid perspective on a linearly ordered set $E$.
Let us denote $t_{i,j}$  the coefficient of $x^iy^j$ in $t(M,N;x,y,1)$.

Then $t_{i,j}$ equals the number of activity classes of reorientations of $M\ra N$ with $i$ dual-active elements and $j$ active elements.
Equally (by Observation \ref{obs:unique-active-fixed}), $t_{i,j}$ equals the number of  reorientations of $M\ra N$,  with $i$ dual-active elements and $j$ active elements, 
which are active-fixed (w.r.t. $M$) and dual-active fixed (w.r.t. $N$).

Furthermore, by combining these counting results, certain remarkable subsets of reorientations are enumerated by Tutte polynomial evaluations as shown in Table \ref{table}. 
\emevder{faut-il ajoutyer precision du genre: Omitting when (dual) active fixed elements are involved amounts to count the coresponding activity classes with the same evaluation. mal dit? integrer dans table? autre table? a prouver ? +++ verifier les (2,1) !}%
In particular, $t(M,N;1,1,1)$ counts the total number of activity classes of reorientations.
\qed
\end{prop}


\begin{table}[H]
\label{table}%
\def\interligne{&\\[-11pt]}%
\vspace{-5mm}
\begin{center}
\begin{tabular}{|l|c|}
\hline
\interligne
\multicolumn{1}{|l}{ 
\begin{tabular}{@{}l@{}}
{\bf reorientations $-_AM\ra -_AN$ with} \\
{\hspace{5mm}\it\small (resp. activity classes with)} 
\end{tabular}  
}
& \multicolumn{1}{|c|}{\bf are counted by}\\
\hline
\interligne
no condition & $t(M,N;2,2,1)$\\
\interligne
acyclic $-_AM$ & $t(M,N;2,0,1)$\\
\interligne
totally cyclic  $-_AN$ & $t(M,N;0,2,1)$\\
 \interligne
 \begin{tabular}{@{}l@{}}
 acyclic $-_AM$ and  dual-active-fixed $-_AN$\\
\hspace{5mm}{\it\small (resp.  acyclic $-_AM$)}
\end{tabular} 
& $t(M,N;1,0,1)$\\
 \interligne
  \begin{tabular}{@{}l@{}}
 active-fixed  $-_AM$  and totally cyclic  $-_AN$ \\
\hspace{5mm}{\it\small (resp. totally cyclic $-_AN$)}
\end{tabular} 
& $t(M,N;0,1,1)$\\
\interligne
 active-fixed $-_AM$ & $t(M,N;2,1,1)$\\
\interligne
  dual-active-fixed $-_AN$ & $t(M,N;1,2,1)$\\
\interligne
  \begin{tabular}{@{}l@{}}
 active-fixed $-_AM$ and   dual-active-fixed $-_AN$\\
\hspace{5mm}{\it\small (resp.  no condition)}
\end{tabular} 
&  $t(M,N;1,1,1)$\\
\interligne
  \begin{tabular}{@{}l@{}}
 acyclic $-_AM$ and  totally cyclic $-_AN$ \\
\hspace{5mm}{\it\small (resp.  acyclic $-_AM$ and  totally cyclic $-_AN$)}
\end{tabular} 
&  $t(M,N;0,0,1)$\\
\hline
\end{tabular}  
\end{center}
\vspace{-5mm}
\caption{Enumeration of certain reorientations based on  representatives of activity classes (Proposition \ref{prop:enum-classes}).
 }
\end{table}


\section{Building on subset activities in matroid perspectives}
\label{sec:subsets}



As explained in the introduction, this section  sums up and completes the construction from \cite{LV13} of the Tutte polynomial of a matroid perspective in terms of subset activities. 
We highlight the close relation between a structural partition and an enumerative Tutte polynomial expansion. Actually, this relation is remarkable in matroids as well.
The reader is referred to \cite{Ox92} for background on matroids, and more specifically to
\cite{Ku86} or \cite[Section 7.3]{Ox92} on perspectives of matroids.
%
Let $M$ and $N$ be two matroids on the set $E$.
The ordered pair $(M,N)$ constitutes a \emph{(matroid) perspective} (also called \emph{quotient map}, or \emph{strong map} up to unimportant variants), and is denoted $M\rightarrow N$, if the following equivalent properties are satisfied:
\begin{itemize}
\itemsep=0mm
\partopsep=0mm 
\topsep=0mm 
\parsep=0mm
\item $r_{N}(X)-r_{N}(Y)\leq r_M(X)-r_M(Y)$ for all $Y\subseteq X\subseteq E$;
\item every flat of $N$ is a flat of $M$;
\item every cocircuit of $N$ is a union of cocircuits of $M$;
\item every circuit of $M$ is a union of circuits of $N$;
\item no circuit of $M$ and cocircuit of $N$ intersect in exactly one element%
\footnote{%
This characterization is mentioned in \cite{LV13}, but not in the usual references on matroid theory \cite{Ku86,Ox92}. We have not found a proof in the literature. Let us give one here.
Assume a circuit $C$ of $M$ and a cocircuit $D$ of $N$ satisfy $C\cap D=\{e\}$. Then, $E\s D$ is a flat of $N$ which is not a flat of $M$. Indeed, $E\s D$ is a flat of $N$ as the complement of a cocircuit, and $E\s D$ is not a flat of $M$, otherwise $C\s\{e\}\subseteq E\s D$ would imply $e\in E\s D$. 
Conversely, assume $F$ is a flat of $N$ which is not a flat of $M$. Then, there exists a  circuit $C$ of $M$ and $e\in C$ such that $e\not\in F$ and $C\s\{e\}\subseteq F$. Since $F$ is a flat of $N$, let us consider a basis $B$ of $N$ formed by a basis of $N(F)$, plus $\{e\}$, plus some elements of $E$. Then, the closure of $B\s\{e\}$ in $N$ contains $F$, and its complement is a cocircuit $D$ of $N$. By construction we have $C\s \{e\}\subseteq F\subseteq E\s D$ and $e\in D$, hence finally $C\cap D=\{e\}$.
}%
;
\item there is a matroid $N$ on $E\uplus A$ such that $M=N\s A$ and $N=N/A$. 
\end{itemize}

Observe that, for a matroid $M$,  $M\ra M$ is a perspective, meaning that every result given for perspectives is also true for matroids by setting $M=N$.
%
Actually, if $M\ra N$ is a matroid perspective and $r(M)=r(N)$ then $M=N$. 
Let us observe that 
if $M\ra N$ is a matroid perspective then $N^*\ra M^*$ is also a matroid perspective.
\bs

\emevder{RK: comme TP deja defini en terms de rang dans intro, ai decide de ne pas le repeter ici... mis en commentaire dans source si besoin}
%

Let $M$ be a matroid on a linearly ordered set $E$. 
For $A\subseteq E$, we define: 
\vspace{-0.2cm}
\begin{eqnarray*}
\Int_{M}(A)& =& \bigl\{e\in A\mid e=\min(C),\ C \text{ cocircuit contained in } (E\s A)\cup\{e\} \bigr\}, \\
\Ext_M(A) & =& \bigl\{e\in E\s A\mid e=\min(C),\ C \text{ circuit contained in } A\cup\{e\} \bigr\}. 
\end{eqnarray*}
\vspace{-5mm}

\noindent Then the parameters $| \Int_{M}(A) |$ and $| \Ext_{M}(A) |$ are respectively called the \emph{internal activity} and \emph{external activity}  of $A$ in $M$.
%
%
%
%
%
%
%
%
If $A$ is a basis of $M$, then the above definitions of $\Int_M(A)$, $\io_M(A)$, $\Ext_M(A)$, and $\ep_M(A)$ are the same as the usual ones for basis activities.
Extending the usual notion of activities from matroid basis to subsets 
%
enables the following expression of the Tutte polynomial of a matroid perspective in terms of independent/spanning subset activities from \cite{LV80, LV99}, which extends the usual one of a matroid in terms of basis activities
(the initial definition  of the Tutte polynomial of a matroid perspective in terms of rank functions is recalled in Section \ref{sec:intro}).


\bs

Let $M\ra N$ be a matroid perspective on a linearly ordered set $E$. 
By \cite{LV80,LV99}, 
its Tutte polynomial  satisfies

$$t(M,N;x,y,z)=\sum_{\substack{B\subseteq E\\ B\text{ independant in }M\\ \text{ and spanning in } N}}\ x^{\mid \Int_N(B)\mid}\ y^{\mid \Ext_M(B)\mid}\ z^{r(M)-r(N)-\bigl(r_M(B)-r_{N}(B)\bigr)}.$$

Beyond this, subset activities 
yield the following remarkable partition into boolean intervals of subsets, 
as shown in 
\cite{LV13}
by applying the general construction of \cite{Da81}.
%
%
%
We have
$$2^E=\biguplus_{\substack{B\subseteq E\\ B\text{ independent in }M\\ \text{ and spanning in } N}}\Bigl[\ B\setminus \Int_{N}(B),\ B\cup \Ext_M(B)\ \Bigr].$$

%

Now, let $M$ be a matroid on a linearly ordered set $E$, for $A\subseteq E$, we define: 
\vspace{-0.2cm}
\begin{eqnarray*}
P_{M}(A) & =& \bigl\{e\in E\s A\mid e=\min(C),\ C \text{ cocircuit contained in } E\s A \bigr\}, \\
Q_M(A) & =& \bigl\{e\in A\mid e=\min(C),\ C \text{ circuit contained in } A \bigr\}. 
\end{eqnarray*}
The cardinalities of these two  sets turn out to be well-known, and independent of the ordering of~$E$: 
\begin{eqnarray*}
|P_M(A)|&=&r(M)-r_M(A),\\
|Q_M(A)|&=&|A|-r_M(A). 
\end{eqnarray*}
%



%

\noindent Finally, subset activities can be seen as parameters that situate the position of a subset inside the boolean interval which contains it with respect to the above partition.
%
Precisely, let $B\subseteq E$ be independent in $M$ and spanning in $N$, and let $A$ be in the interval $[B\s \Int_{N}(B), B\cup \Ext_M(B)]$. 
We have the following relations, showing the consistency of the whole above setting:
\begin{eqnarray*}
\Int_{N}(A)& =& \Int_{N}(B)\cap A, \\
P_{N}(A) & =& \Int_{N}(B)\s A, \\
\Ext_M(A) & =& \Ext_M(B)\s A, \\
Q_M(A) & =& \Ext_M(B)\cap A. 
\end{eqnarray*}


More properties of the above partition into intervals, and, more specifically, of the relations between an independent/spanning subset and the subsets in its associated interval, are listed in \cite[Prop. 4.4]{LV13}.
However, one non-trivial property is missing in this list and never mentioned in \cite{LV13}, whereas it is required  
to prove the main theorem%
\footnote{Precisely: \cite[Lemma 4.6]{LV13} actually follows from \cite[Prop. 4.4]{LV13} as written in \cite{LV13} but only once the  argument of Proposition  \ref{lem:rcd} is established, ensuring that, in the statement of \cite[Lemma 4.6]{LV13}, we have  $rcd_{M,M'}(A)=rcd_{M,M'}(B)$, so that the parameter $k$ makes sense  and is correctly handled, 
.}%
. 
The next proposition fills this gap,
and it 
is also required  in our proof of the main theorem below.

\begin{prop}
\label{lem:rcd}
Let $M\ra N$ be a matroid perspective on a linearly ordered set $E$. Let $B\subseteq E$ be independent in $M$  and spanning in $N$, and let $A\in\bigl[\ B\setminus \Int_{N}(B),\ B\cup \Ext_M(B)\ \bigr]$.
Then we have $$r_M(A)-r_N(A)=r_M(B)-r_N(B).$$
\end{prop}

\begin{proof}
We have $A=B\s P \cup Q$ with $P\subseteq \Int_{N}(B)\subseteq B$ and $Q\subseteq \Ext_M(B)\subseteq E\s B$.
First, let us consider any $e\in Q$ and $f\in P$.
By definition of $\Ext_M(B)$, there exists a circuit $C$ of $M$ with $C\s \{e\}\subseteq B$ and $e=\min(C)$.
By definition of $\Int_{N}(B)$, there exists a cocircuit $D$ of $N$ with $D\s\{f\}\subseteq E\s B$ and $f=\min(D)$.
By construction, we have $C\cap D\subseteq \{e,f\}$.
Let us prove that $C\cap D=\emptyset$.
Otherwise, by properties of a matroid perspective (no circuit of $M$ and cocircuit of $N$ intersect in exactly one element), we must have $e\not=f$ and $C\cap D=\{e,f\}$ .
In this case, we have $e\leq f$ by properties of $e$, and $f\leq e$ by properties of $f$, hence $e=f$, a contradiction.
So $C\cap D=\emptyset$.
This conclusion is true for any choice of $e$ and $f$, so we have in addition $C\cap P=\emptyset$ and $D\cap Q=\emptyset$.

Now, let us place in $M$.
Since $B$ is independent in $M$, we have $r_M(B)=|B|$.
Moreover, we have that $B\s P$ is independent in $M$. 
For $e\in Q$ and a circuit $C$ of $M$ as above, 
we have shown that $C\cap P=\emptyset$, hence $C\s\{e\}\subseteq B\s P$, hence $e$ belongs to the closure of $B\s P$ in $M$. This is true for any $e\in Q$, so $A=B\s P \cup Q$ is contained in the closure of $B\s P$ in $M$.
So we have $r_M(A)=r_M(B\s P)$. Since $P\subseteq B$ and $B\s P$ is independent, we finally have $r_M(A)=|B|-|P|$ .

Now, let us place in $N$.
Since $B$ is spanning in $N$, we have $r_N(B)=r(N)$.
For $f\in P$ and a cocircuit $D$ of $N$ as above,
we have $D\s\{f\}\subseteq E\s B$, hence $B\s \{f\}$ spans the hyperplane $E\s D$ of $N$ (flat with rank $r(N)-1$).
Moreover, we have shown that $D\cap Q=\emptyset$, that is: $Q$ is contained in this hyperplane $E\s D$ of $N$. All this is true for any $f\in P$, so let us consider the intersection of the associated hyperplanes $E\s D$ when $f$ ranges $P$. We get that the flat spanned by $B\s P$ has rank $r(N)-|P|$ and that $Q$ is contained in this flat.
So we have shown that $r_N(A)=r(N)-|P|$.

Finally, we have $r_M(A)-r_N(A)=|B|-|P|-(r(N)-|P|)=|B|-r(N)=r_M(B)-r_N(B)$.
\end{proof}

Now, let us consider the main theorem below. This general Tutte polynomial expansion formula, known for matroids from \cite{GoTr90} and generalized to matroid perspectives in \cite{LV13}, contains a lot of formulas by suitable variable choices, see \cite{GoTr90, LV13}. Notably, a charm of this formula is that it contains the two formulas in three variable previously given: the first (in terms of rank function) by using the variables $(1,x-1,1,y-1,z)$, the second (in terms of activities for independent/spanning subsets) by using the variables $(x,0,y,0,z)$.
We give below a short natural proof, 
highlighting how this formula is the enumerative counterpart of the fundamental structural partition into intervals addressed above.
The proof in \cite{LV13}  is in terms of expressions of derivatives of the Tutte polynomial, which is indirectly equivalent by means of Taylor's theorem, but which is not combinatorially striking the same way.

%
%
%

\ss

\begin{thm}[{\cite{LV13}}]
\label{th:Tutte-4-variables-perspective}
Let $M\ra N$ be a matroid perspective on a linearly ordered set $E$. We have
$$t(M,N;x+u,y+v,z)=\sum_{A\subseteq E}\ x^{\mid \Int_N(A)\mid}\ u^{\mid P_N(A)\mid}\ y^{\mid \Ext_M(A)\mid}\ v^{\mid Q_M(A)\mid}\ z^{r(M)-r(N)-\bigl(r_M(A)-r_{N}(A)\bigr)}$$
\end{thm}



\begin{proof} 
\parskip=2mm
\noindent
For $A\subseteq E$, let us denote $rcd_{M,N}(A)=r(M)-r(N)-\bigl(r_M(A)-r_{N}(A)\bigr)$. 
%
By the expression of $t(M;x+u,y+v)$ recalled above, we have

\noindent$\displaystyle t(M,N;x+u,y+v,z)=\sum_{\substack{B\subseteq E\\ B\text{ independant in }M\\ \text{ and spanning in } N}}\ (x+u)^{\mid \Int_N(B)\mid}\ (y+v)^{\mid \Ext_M(B)\mid}\ \ z^{rcd_{M,N}(B)}.$

%
%
\noindent By the binomial formula, this expression equals

\noindent $\displaystyle  \sum_{\substack{B\subseteq E\\ B\text{ independant in }M\\ \text{ and spanning in } N}}\ 
\Bigl(\ \sum_{A'\subseteq \Int_N(B)} x^{\mid A'\mid}\ u^{\mid \Int_N(B)\s A'\mid}\ \Bigr)\ 
\Bigl(\ \sum_{A''\subseteq \Ext_M(B)} y^{\mid \Ext_M(B)\s A''\mid}\ v^{\mid A''\mid}\ \Bigr)\ z^{rcd_{M,N}(B)}.$


\noindent Since $Int_N(B)\subseteq B$ and $Ext_M(B)\subseteq E\s B$, we have  $\Int_N(B)\cap \Ext_M(B)=\emptyset$. So, one has a bijection between pairs $(A',A'')$ involved in the above expression and subsets $A=A'\uplus A''$ of $\Int_N(B)\uplus \Ext_M(B)$, hence this expression equals

\noindent $\displaystyle  \sum_{\substack{B\subseteq E\\ B\text{ independant in }M\\ \text{ and spanning in } N}}\ 
\Bigl(\ \sum_{A \subseteq \Int_N(B)\uplus \Ext_M(B)} x^{\mid \Int_N(B)\cap A\mid}\ u^{\mid \Int_N(B)\s A\mid}\  y^{\mid \Ext_M(B)\s A\mid}\ v^{\mid \Ext_M(B)\cap A\mid}\ \Bigr)\ z^{rcd_{M,N}(B)}.$


\noindent
Since 
$\bigl(\ B\setminus \Int_{N}(B)\ \bigr)\cap\bigl(\ \Int_N(B)\uplus \Ext_M(B)\ \bigr)=\emptyset$, 
the mapping $A\mapsto A\cup\bigl(\ B\setminus \Int_{N}(B)\ \bigr)$ yields an isomorphism between 
the two boolean intervals $[\emptyset,\ \Int_N(B)\uplus \Ext_M(B)]$ and
$[B\s \Int_{N}(B), B\cup \Ext_M(B)]$, 
which does not change the sets $\Int_N(B)\cap A$, $\Int_N(B)\s A$, $\Ext_M(B)\s A$, and $\Ext_M(B)\cap A$. 
So the above expression can be equivalently written


\noindent $\displaystyle  \sum_{\substack{B\subseteq E\\ B\text{ independant in }M\\ \text{ and spanning in } N}}\ 
\Bigl(\ \sum_{\substack{A \in \bigl[B\s \Int_{N}(B),\\ \hphantom{A \in \bigl[}B\cup \Ext_M(B)\bigr]}} x^{\mid \Int_N(B)\cap A\mid}\ u^{\mid \Int_N(B)\s A\mid}\  y^{\mid \Ext_M(B)\s A\mid}\ v^{\mid \Ext_M(B)\cap A\mid}\ \Bigr)\ z^{rcd_{M,N}(B)}.$

%

\noindent By Proposition \ref{lem:rcd}, we have $rcd_{M,N}(B)=rcd_{M,N}(A)$ for $B$ independent in $M$ and spanning in $N$ and for $A\in [B\s \Int_{N}(B), B\cup \Ext_M(B)]$. Hence the above expression equals

\noindent $\displaystyle  \sum_{\substack{B\subseteq E\\ B\text{ independant in }M\\ \text{ and spanning in } N}}\ 
\sum_{\substack{A \in \bigl[B\s \Int_{N}(B),\\ \hphantom{A \in \bigl[}B\cup \Ext_M(B)\bigr]}} x^{\mid \Int_N(B)\cap A\mid}\ u^{\mid \Int_N(B)\s A\mid}\  y^{\mid \Ext_M(B)\s A\mid}\ v^{\mid \Ext_M(B)\cap A\mid}\ z^{rcd_{M,N}(A)}.$

\noindent Since $2^E$ is the disjoint union of intervals $[B\s \Int_{M}(B), B\cup \Ext_M(B)]$ for $B$ independent in $M$ and spanning in $N$ (as recalled above), this expression equals
%

\noindent $\displaystyle  \sum_{A\subseteq E}\ x^{\mid \Int_N(B)\cap A\mid}\ u^{\mid \Int_N(B)\s A\mid}\  y^{\mid \Ext_M(B)\s A\mid}\ v^{\mid \Ext_M(B)\cap A\mid}\ z^{rcd_{M,N}(A)}.$

\noindent Finally, replacing the exponents as
recalled above, this expression equals the required one.
\end{proof}

\bibliographystyle{alpha}
\small
\bibliography{references-expansions}

\newcommand{\etalchar}[1]{$^{#1}$}
\begin{thebibliography}{BGdON09}

\bibitem[BGdON09]{BrGONo09}
J.M. Brunat, A.~Guedes~de Oliveira, and M.~Noy.
\newblock Partitions of a finite boolean lattice into intervals.
\newblock {\em European J. Combin.}, 30:1801–--1809, 2009.

\bibitem[Bj{\"o}92]{Bj92}
A.~Bj{\"o}rner.
\newblock Homology an shellability of matroids and geometric lattices.
\newblock In N.~White, editor, {\em Matroid applications}, volume~40 of {\em
  Encyclopedia of Mathematics and Its Applications}. Cambridge University
  Press, 1992.

\bibitem[BLVS{\etalchar{+}}99]{OM99}
A.~Bj{\"o}rner, M.~Las~Vergnas, B.~Sturmfels, N.~White, and G.~Ziegler.
\newblock {\em Oriented Matroids}, volume~46 of {\em Encyclopedia of
  Mathematics and Its Applications}.
\newblock Cambridge University Press, second edition, 1999.

\bibitem[Cra69]{Cr69}
H.H. Crapo.
\newblock The tutte polynomial.
\newblock {\em Aequationes Mathematicae}, 3:211--229, 1969.

\bibitem[Daw81]{Da81}
J.E. Dawson.
\newblock A construction for a family of sets and its application to matroids.
\newblock In Springer, editor, {\em Lect. Notes in Math. (Comb. Math. VIII,
  Gelong, 1980)}, volume 884, pages 136--147, 1981.

\bibitem[Gioa]{GiChapterPerspectives}
E.~Gioan.
\newblock The {T}utte polynomial of matroid perspectives.
\newblock In J.~Ellis-Monaghan and I.~Moffatt, editors, {\em Handbook of the
  Tutte Polynomial}, CRC Monographs and Research Notes in Mathematics.
\newblock Submitted.

\bibitem[Giob]{GiChapterOriented}
E.~Gioan.
\newblock The {T}utte polynomial of oriented matroids.
\newblock In J.~Ellis-Monaghan and I.~Moffatt, editors, {\em Handbook of the
  Tutte Polynomial}, CRC Monographs and Research Notes in Mathematics.
\newblock Submitted.

\bibitem[Gio02]{Gi02}
E.~Gioan.
\newblock {\em Correspondance naturelle entre bases et r\'eorientations des
  matro{\"\i}des orient\'es}.
\newblock PhD thesis, University of Bordeaux 1, 2002.
\newblock (available at {\tt http://www.lirmm.fr/~gioan}).

\bibitem[GLVa]{AB2-b}
E.~Gioan and M.~Las~Vergnas.
\newblock The active bijection - 2.b - {D}ecomposition of activities for
  oriented matroids, and general definitions of the active bijection.
\newblock {\em Submitted, preprint available at arXiv:1807.06578}.

\bibitem[GLVb]{ABG2}
E.~Gioan and M.~Las~Vergnas.
\newblock The active bijection for graphs.
\newblock {\em Submitted, preprint available at arXiv:1807.06545}.

\bibitem[GLV05]{GiLV05}
E.~Gioan and M.~Las~Vergnas.
\newblock Activity preserving bijections between spanning trees and
  orientations in graphs.
\newblock {\em Discrete Mathematics}, 298:169--188, 2005.

\bibitem[GLV07]{GiLVEurocomb07}
E.~Gioan and M.~Las~Vergnas.
\newblock Fully optimal bases and the active bijection in graphs, hyperplane
  arrangements, and oriented matroids.
\newblock {\em Electronic Notes in Discrete Mathematics}, 29:365--371, 2007.
\newblock Proceedings EuroComb 2007 (Sevilla).

\bibitem[GM97]{GoMM97}
G.~Gordon and E.~McMahon.
\newblock Interval partitions and activities for the greedoid tutte polynomial.
\newblock {\em Adv. Appl. Math.}, 18:33--49, 1997.

\bibitem[GT90]{GoTr90}
G.~Gordon and L.~Traldi.
\newblock Generalized activities and the {T}utte polynomial.
\newblock {\em Disc. Math.}, 85:167--176, 1990.

\bibitem[Kun86]{Ku86}
J.~Kung.
\newblock Strong maps.
\newblock In N.~White, editor, {\em Theory of matroids}, volume~26 of {\em
  Encyclopedia of Mathematics and Its Applications}. Cambridge University
  Press, 1986.

\bibitem[LV80]{LV80}
M.~Las~Vergnas.
\newblock On the {T}utte polynomial of a morphism of matroids.
\newblock {\em Annals of Discrete Mathematics}, 8:7--20, 1980.

\bibitem[LV84]{LV84}
M.~Las~Vergnas.
\newblock The {T}utte polynomial of a morphism of matroids {II.} {A}ctivities
  of orientations.
\newblock In J.A. Bondy and U.S.R. Murty, editors, {\em Progress in Graph
  Theory}, pages 367--380. Academic Press, Toronto, Canada, 1984.
\newblock (Proc. Waterloo Silver Jubilee Conf. 1982).

\bibitem[LV99]{LV99}
M.~Las~Vergnas.
\newblock The {T}utte polynomial of a morphism of matroids {I.} {S}et-pointed
  matroids and matroid perspectives.
\newblock {\em Ann. Inst. Fourier, Grenoble}, 49(3):973--1015, 1999.

\bibitem[LV01]{LV01}
M.~Las~Vergnas.
\newblock Active orders for matroid bases.
\newblock {\em Europ. J. Comb.}, 22:709--721, 2001.

\bibitem[LV12]{LV12}
M.~Las~Vergnas.
\newblock The {T}utte polynomial of a morphism of matroids {6.} {A}
  multi-faceted counting formula for hyperplane regions and acyclic
  orientations.
\newblock 2012.
\newblock Unpublished, preliminary preprint available at arXiv 1205.5424.

\bibitem[LV13]{LV13}
M.~Las~Vergnas.
\newblock The {T}utte polynomial of a morphism of matroids {5.} {D}erivatives
  as generating functions of {T}utte activities.
\newblock {\em Europ. J. Comb.}, 34:1390--1405, 2013.

\bibitem[Oxl11]{Ox92}
J.G. Oxley.
\newblock {\em Matroid Theory}.
\newblock Oxford Graduate Texts in Mathematics. Oxford University Press, 2011.
\newblock Second edition (first edition in 1992).

\bibitem[Tut54]{Tu54}
W.T. Tutte.
\newblock A contribution to the theory of chromatic polynomials.
\newblock {\em Canadian J. Math.}, 6:80--91, 1954.

\end{thebibliography}

\end{document}